\documentclass[11pt]{amsart}
\usepackage{amscd,amssymb,amsmath}
 \newtheorem{theorem}{Theorem}[section]

 \theoremstyle{definition}

 \newtheorem{proposition}[theorem]{Proposition}
 \newtheorem{corollary}[theorem]{Corollary}

 \theoremstyle{remark}

 \numberwithin{equation}{section}
\begin{document}

\title[Involutions on Sol 3-manifolds]{Involutions on sapphire Sol 3-manifolds and the Borsuk-Ulam theorem for maps into $R^n$}



  \author[A. P. Barreto]{Alexandre Paiva Barreto}
 \address{Departamento de Matem\'atica\\
 Universidade Federal de S\~ao Carlos\\
 S\~ao Carlos, 13565-905, Brazil}
 \email{alexandre@dm.ufscar.br}
\thanks{This work was initiated during the fixed point conference in China June 2011.  This project was supported in part by Projeto Tem\'atico Topologia Alg\'ebrica Geom\'etrica e Diferencial 2012/24454-8.}

 \author[D. L. Gon\c calves]{Daciberg Lima Gon\c calves}
 \address{Departamento. de Matem\'atica\\
  - IME - USP, Caixa Postal 66.281\\
   CEP 05314- 970,FAX: 55-11-30916183\\
     S\~ao Paulo - SP, Brasil }
 \email{dlgoncal@ime.usp.br}

 \author[D. Vendr\'uscolo]{Daniel Vendr\'uscolo}
  \address{Departmento de Matem\'atica\\
 Universidade Federal de S\~ao Carlos\\
 S\~ao Carlos, 13565-905, Brazil}
 \email{daniel@dm.ufscar.br}

 \newcommand{\af}{\alpha}
 \newcommand{\et}{\eta}
 \newcommand{\ga}{\gamma}
 \newcommand{\ta}{\tau}
 \newcommand{\ph}{\varphi}
 \newcommand{\bt}{\beta}
 \newcommand{\lb}{\lambda}
 \newcommand{\wh}{\widehat}
 \newcommand{\sg}{\sigma}
 \newcommand{\om}{\omega}
 \newcommand{\cH}{\mathcal H}
 \newcommand{\cF}{\mathcal F}
 \newcommand{\N}{\mathcal N}
 \newcommand{\R}{\mathbb R}
 \newcommand{\Z}{\mathbb Z}
 \newcommand{\Ga}{\Gamma}
 \newcommand{\cc}{\mathcal C}
 \newcommand{\bea} {\begin{eqnarray*}}
 \newcommand{\beq} {\begin{equation}}
 \newcommand{\bey} {\begin{eqnarray}}
 \newcommand{\eea} {\end{eqnarray*}}
 \newcommand{\eeq} {\end{equation}}
 \newcommand{\eey} {\end{eqnarray}}
 \newcommand{\ovl}{\overline}
 \newcommand{\vv}{\vspace{4mm}}
 \newcommand{\lra}{\longrightarrow}

\begin{abstract}
  For each sapphire Sol 3-manifold, which is not a torus bundle over the circle, we classify the free  involutions. Then we classify the  triple $(M, \tau; R^n)$ where $M$ is a sapphire Sol 3-manifold as above, $\tau$ is a free involution and $n$ a positive integer, for which the Borsuk-Ulam property holds. It is known that for $n>3$, the validity of the Borsuk-Ulam property  is  of the involution, so the main classification is for $n=2$ and $3$.
\end{abstract}

\maketitle

\section{Introduction}

The classical Borsuk-Ulam Theorem states that, for any continuous map $f:S^n\to \R^n$, there exists a point $x\in S^n$ such that $f(x)=f(-x)$. This theorem has motivated the following quite natural and general question.  Given a topological space $M$, a free involution $\tau$ on $M$, and a positive  integer,  we say that {\it the Borsuk-Ulam property holds for the  triple $(M,\tau,\R^n)$ ( or the  triple $(M,\tau,\R^n)$ satisfies the  Borsuk-Ulam property),  if for any continuous map $f:M\to \R^n$, there exists a point $x\in X$ such that $f(x)=f(\tau(x))$}. The question consists in classifying the triples  $(M,\tau,\R^n)$ for which  the  Borsuk-Ulam property holds. 

 The above  question has been studied by several people, see for example \cite{Gon, GonGua, GSM} among others. The classification of the free involutions of a space   is closely related to the above question above and  is a problem in its own right. For many  Seifert $3$-manifolds, the classification of involutions, together with the study of the Borsuk-Ulam property can be found in  \cite{HaGoZv},  \cite{BauHaGoZv1} and  \cite{BauHaGoZv2}.

The purpose of this work is to classify all free involutions  on a closed sapphire Sol manifold $M$, up to an equivalence relation (see section $3$), as well the  triples $(M, \tau, \R^n)$ which satisfy  the  Borsuk-Ulam property. 

The main results are:
 
{\bf Proposition} \ref{r2}. Any triple  $(M, \tau; \R^2)$ has the Borsuk-Ulam property.
 
{\bf Theorem} \ref{main}.  Given a sapphire 3-manifold determined by the matrix

$$\begin{pmatrix}
           a & b \\
           c & d
           \end{pmatrix}$$
\noindent we have:

I) If $c$ is odd the manifold does not admit involutions.

II) If $c,b$ are even we have two cases:

II-a) $|a|\neq |d|$, then the manifold admits one class of free involutions and the quotient is homeomorphic to the sapphire determined by the matrix
$$\begin{pmatrix}
           a & 2b \\
         c/2 & d
           \end{pmatrix}.$$

II-b) $|a|=|d|$, then the manifold admits 3 distinct classes of involutions and the quotients are the sapphires manifolds given by the matrices:
$$\begin{pmatrix}
           a & 2b \\
         c/2 & d,
           \end{pmatrix},\qquad 
           \begin{pmatrix}
           r & s \\
           t & u,
           \end{pmatrix},\qquad
           \begin{pmatrix}
           s & r \\
           u & t
           \end{pmatrix},$$
\noindent where $(r,s,t,u)$ is one of the solutions given by
Proposition \ref{SapphireDown1}, if  $r\ne u$ and $s\ne t$. Otherwise  the manifold admits at least three distinct classes of involutions and 
at most five.

III) If $c$ is even and $b$ odd, then the manifold admits one class of free involutions and the quotient is homeomorphic to the sapphire determined by the matrix
$$\begin{pmatrix}
           a & 2b \\
         c/2 & d
           \end{pmatrix}.$$

{\bf Corollary} \ref{bumain}.  Given a sapphire 3-manifold $S_A$ determined by the matrix

$$\begin{pmatrix}
           a & b \\
           c & d
           \end{pmatrix}$$
\noindent the Borsuk-Ulam property holds for the triple $(S_A, \tau, \R^3)$ if and only if $c$ is even and $b$ is odd.
Furthermore, in this case the manifold admits only one class of free involutions.  

This paper contains 3 sections besides this introduction.  
 
In section 2 we recall the definition and a classification for the closed sapphire Sol 3-manifolds, as well those which are not torus bundles. They are classified in terms of certain integral matrices in $GL(2,\Z)$ by a result of  K. Morimoto in \cite{morimoto}. 

In section 3, we determine all the double coverings of a given sapphire Sol manifold in terms of the classification given by Morimoto. Then we obtain the equivalence classes of  the double coverings.   

In section 4 we provide the classification of the  involutions. Then we give the   classification of all triples  $(M, \tau, \R^2)$ and  $(M, \tau, \R^3)$ for which  the  Borsuk-Ulam property holds. The latter case uses a recent result by J. Hillman, \cite{hillman}.

A great amount of information about involutions (not necessarily free) on torus bundles is given  in the work of Sakuma \cite{sakuma}. That work might be useful to study similar  questions, such  as that studied  here,  for torus bundles.  This is  work in progress and it is not clear what type of results will appear, especially if we wish to obtain  answers in the same spirit as those given in this paper.  It is likely that they will  not be  so explicit as those  for the  sapphire.

To conclude we would like  to comment that for some special Sol manifolds, the manifold admits two involutions, and  we do not know if they are equivalent. See more about this at the end of section $3$.
 
\section{Sol $3$-manifolds}

The Sol 3-manifolds constitutes one of the eight geometries mentioned in the geometrization conjecture of Thurston. Using either \cite{SWW} or \cite{morimoto}  we have that these manifolds can be divided into two disjoint subfamilies:
\begin{itemize}
\item[a)] Torus bundles where the gluing map is an Anosov map;
\item[b)] A subset of the sapphire manifolds (called torus semi-bundles in \cite{SWW}) 
which are  those where the gluing automorphisms are the automorphism of $\Z+\Z$ where $rstu\ne 0$ (see the definition below).
\end{itemize}

In the case of torus bundles, two of them are homeomorphic if and only if the gluing matrices are conjugated or one is conjugated to the inverse of the other, see \cite{sakuma}. Let
$A=\begin{pmatrix}
          m & n \\
          p & q
        \end{pmatrix}$ be the matrix of the Anosov map used to construct the torus bundles $T_A$, so the fundamental group of such Sol 3-manifold admits the following presentation:
\begin{equation}\label{pi1torusbundle}
\pi_1(T_A)\cong <a,b,c\ |\ aba^{-1}b^{-1}, cac^{-1}b^{-p}a^{-m}, cbc^{-1}b^{-q}a^{-n}>.
\end{equation}
\noindent For the above  facts  and more details about torus bundles, see \cite{sakuma}.

In the case of Sol $3$-manifolds which are not torus bundles, recall from K. Morimoto \cite{morimoto} that a {\it sapphire} manifold is obtained by gluing two orientable 
twisted $I$-bundles over the Klein bottle as follows. For $i=1,2$, let $K_iI$ be two copies of the same orientable twisted $I$-bundles over the Klein bottle with $\pi_1(\partial K_iI)\cong \langle \af_i,\bt_i|\af_i\bt_i=\bt_i\af_i\rangle$.
Let $A=\begin{pmatrix}
          r & s \\
          t & u
        \end{pmatrix}$ be an element of $GL_2(\Z)$ and $\phi:\partial K_2I\to \partial K_1I$ be a homeomorphism that induces isomorphism $A=\phi_{\#}:\pi_1(\partial K_2I)\to \pi_1(\partial K_1I)$ so that $\phi_{\#}(\af_2)=r\af_1+s\bt_1$ and $\phi_{\#}(\bt_2)=t\af_1+u\bt_1$.
By identifying $x\in \partial K_2I$ with $ \phi(x) \in \partial K_1I$, we obtain a closed orientable $3$-manifold $S_{A}=K_1I\cup_\phi K_2I$, wich is the {\it sapphire}. The sapphire manifolds which are Sol 3-manifolds are, by  Proposition 1.5 \cite{SWW} (or Remmark 1.8 \cite{morimoto}) correspond to those for which  the gluing matrix has all entries non zero, i.e.   $rstu\ne 0$.

 A presentation of the fundamental group is given by Morimoto: 
\begin{equation}\label{pi1sapphire}
\pi_1(S_{A})\cong \langle a,b,c\ |aba^{-1}b= c^{2}a^{-2r}b^{-s}=ca^{2t}b^{u}c^{-1}a^{2t}b^{u}=1 \rangle,
\end{equation}
so we can also compute the first homology group. 

 In \cite{morimoto} Morimoto also describes when two gluing matrices produce the same manifold up to homeomorphisms. The topological types of sapphire spaces are given by Theorem 1 in \cite{morimoto}, namely:

\begin{theorem}[Morimoto]\label{morimoto} Let $A$, $A'$ be two elements of $GL(2,\Z)$. Then $N_A$ is homeomorphic to $N_{A'}$  if and only if $A'$ is equal to one of the following matrices: $\pm  A^{\pm}$, $\pm  B  A^{\pm}$ $\pm  A^{\pm}B$, $\pm  BA^{\pm}B$, where
 $B=\begin{pmatrix}
          1 & 0 \\
          0 & -1
        \end{pmatrix}.$
\end{theorem}

 By abuse of notation, from now on we say that {\it a manifold is sapphire if it is a Sol-manifold}.
  
Using the classification given by the above theorem, if $S_A$ is a sapphire manifold  we may suppose that $A=\begin{pmatrix}
          r & s \\
          t & u
        \end{pmatrix}$ is such that all four numbers $r,s,t $ and $u$ are positive, and moreover $r\le u$.

The first homology group with integral coefficients $\Z$ of a sapphire may be obtained from the above presentation,  and it is explicitly  given  in \cite[Proposition 1.6]{morimoto}. The   family of the sapphires (which are not torus bundles and have geometry Sol), can be divided into two subfamilies. In the first subfamily we consider the sapphires having $H_1$(first homology group with integral coefficients)  isomorphic to $\Z_{4t}+\Z_4$. This corresponds to  $s$ being odd in the gluing map. In the second subfamily, we consider the sapphire having $H_1$  isomorphic to $\Z_{4t}+\Z_2+\Z_2$. This corresponds to $s$ being even.

\section{The classification of the double covering of a sapphire Sol manifold $M$}

In this section,  we will study the double covering of a sapphire Sol manifold. We will describe the double coverings using the kernel of nontrivial epimorphisms from $\pi_1(M)$ to $\mathbb{Z}_2$. In terms of the presentation given in section \ref{pi1sapphire},  we have to describe the image of $a$, $b$ and $c$, so we will have 7 cases:

\subsubsection*{Case I: $\varphi_1(a)=\overline{1}$, $\varphi_1(b)=\overline{0}$, $\varphi_1(c)=\overline{0}$}

Making use  of  the Reidemeister-Schreier method, we take $\{1,a\}$ as a Schreier system for right cosets of the kernel ($H$) of $\varphi_1$,  we obtain  $\{b,c,a^{2},aba^{-1}, aca^{-1}\}$ as a system of generators and the following relations:

\begin{itemize}
\item $aba^{-1}b$;
\item $c^{2}a^{-2r}b^{-s}$;
\item $ca^{2t}b^{u}$;
\item $ba^{-1}ba$;
\item $ a^{-1}c^{2}a^{-2r}b^{-s}a$;
\item $ a^{-1}ca^{2t}b^{u}c^{-1}a^{2t}b^{u}a$.
\end{itemize}

Now we make $\alpha=b,\beta=c,\gamma=a^{2},\delta=aba^{-1},\lambda=aca^{-1}$ and we obtain

$$H=\left\langle
\begin{array}
[c]{c}%
 \alpha,\beta,\gamma,\delta,\lambda\ |\ \delta\alpha=\beta
^{2}\gamma^{-r}\alpha^{-s}=\beta\gamma^{t}\alpha^{u}\beta^{-1}\gamma^{t}%
\alpha^{u}=\alpha\gamma^{-1}\delta\gamma=\\
=\lambda^{2}\gamma^{-r}\delta
^{-s}=\lambda\gamma^{t}\delta^{u}\lambda^{-1}\gamma^{t}\delta^{u}%
=1
\end{array}
\right\rangle.$$

Setting $a_1=\beta$, $b_1=\alpha^{u}\gamma^{t}$ and $c_1=\lambda$ we will have:
$$ H=< a_1,b_1,c_1\ |\ a_1b_1a_1^{-1}b_1=c_1^{2}a_1^{-2R}b_1^{S}= c_1a_1^{2T}b_1^{U}c_1^{-1}a_1^{2T}b_1^{U}=1>$$
where
\begin{itemize}
\item $R=ru+st,\ S=-2rs,\ T=2tu,\ U=-(ru+st)$ if $ru-st=1$;
\item $R=-(ru+st),\ S=2rs,\ T=-2tu,\ U=ru+st$ if $ru-st=-1$.
\end{itemize}

So the double covering defined by $\varphi_1$ is the sapphire determined by the matrix
$\begin{pmatrix}
          ru+st & -2rs \\
          2tu & -(ru+st)
        \end{pmatrix}$
if $ru-st=1$, or the sapphire determined by
$\begin{pmatrix}
          -(ru+st) & 2rs \\
          -2tu & ru+st
        \end{pmatrix}$ if $ru-st=-1$. Moreover it is clear that both matrices describe the same 3-manifold.

\subsubsection*{Case II: $\varphi_2(a)=\overline{0}$, $\varphi_2(b)=\overline{1}$, $\varphi_2(c)=\overline{0}$}

First of all we note that in this case (and in all other cases with nontrivial image of $b$),  $s$ must be even and $u$ odd, we denote $s=2k$ and $u=2l-1$. Taking $\{1,b\}$ as a Schreier system for right cosets of the kernel ($H$) of $\varphi_2$,   we obtain  $\{a,c,bab^{-1},b^{2},bcb^{-1}\}$ as a system of generators and the following relations:

\begin{itemize}
\item $aba^{-1}b$;
\item $c^{2}a^{-2r}b^{-s}$;
\item $ ca^{2t}b^{u}c^{-1}a^{2t}b^{u} $;
\item $ b^{-1}aba^{-1}b^{2} $;
\item $ b^{-1}c^{2}a^{-2r}b^{-s+1} $;
\item $b^{-1}ca^{2t}b^{u}c^{-1}a^{2t}b^{u+1}  $.
\end{itemize}

Taking  $\alpha=a,\beta=c,\gamma=bab^{-1},\delta=b^{2},\lambda=bcb^{-1}$,   we obtain
$$H=\left\langle
\begin{array}
[c]{c}%
\alpha,\beta,\gamma,\delta,\lambda\ | \ \alpha\gamma^{-1}
\delta=\beta^{2}\alpha^{-2r}\delta^{-k}=\beta\alpha^{2t}\delta^{l-1}
\lambda^{-1}\gamma^{2t}\delta^{l}=\gamma\delta\alpha^{-1}=\\
=\lambda^{2} \gamma^{-2r}
\delta^{-k}=\lambda\gamma^{2t}\delta^{l}\beta^{-1}\alpha^{2t}\delta^{l-1}=1
\end{array}
\right\rangle.$$

Setting $a_1=\alpha$, $b_1=\delta$ and $c_1=\lambda$ we will have:
$$ H=< a_1,b_1,c_1\ |\ a_1b_1a_1^{-1}b_1=c_1^{2}a_1^{-2R}b_1^{S}= c_1a_1^{2T}b_1^{U}c_1^{-1}a_1^{2T}b_1^{U}=1>$$
where $R=r, S=k=\dfrac{s}{2}, T=2t$ and $U=u$. 

So the double covering defined by $\varphi_2$ is the sapphire determined by the matrix:
$$\begin{pmatrix}
          r & \dfrac{s}{2} \\
          2t & u
        \end{pmatrix}.$$

\subsubsection*{Case III: $\varphi_3(a)=\overline{0}$, $\varphi_3(b)=\overline{0}$, $\varphi_3(c)=\overline{1}$}

Taking $\{1,c\}$ as a Schreier system for right cosets of the kernel ($H$) of $\varphi_2$,  we obtain  $\{a,b,cac^{-1},cbc^{-1},c^{2}\}$ as a system of generators,  and the following relations:

\begin{itemize}
\item $aba^{-1}b$;
\item $c^{2}a^{-2r}b^{-s}$;
\item $ ca^{2t}b^{u}c^{-1}a^{2t}b^{u} $;
\item $c^{-1}aba^{-1}bc $;
\item $ ca^{-2r}b^{-s}c $;
\item $a^{2t}b^{u}c^{-1}a^{2t}b^{u}c  $.
\end{itemize}

Taking  $\alpha=a,\beta=b,\gamma=cac^{-1},\delta=cbc^{-1},\lambda=c^{2}$,  we obtain
$$H=\left\langle
\begin{array}
[c]{c}%
\alpha,\beta,\gamma,\delta,\lambda\;;\;\alpha\beta\alpha^{-1}\beta=\lambda\alpha^{-2r}\beta^{-s}=\gamma^{2t}\delta^{u}\alpha^{2t}
\beta^{u}=\gamma\delta\gamma^{-1}\delta=\\
=\gamma^{-2r}\delta^{-s}\lambda=\alpha^{2t}\beta^{u}\lambda^{-1}\gamma^{2t}\delta^{u}\lambda=1
\end{array}
\right\rangle.$$

Setting $a_1=\alpha$, $b_1=\beta$ and $c_1=\gamma$, we   have:
$$ H=< a_1,b_1,c_1\ |\ a_1b_1a_1^{-1}b_1=c_1^{2}a_1^{-2R}b_1^{S}= c_1a_1^{2T}b_1^{U}c_1^{-1}a_1^{2T}b_1^{U}=1>$$
where
\begin{itemize}
\item $R=ru+st,\ S=2su,\ T=-2rt,\ U=-(ru+st)$ if $ru-st=1$;
\item $R=-(ru+st),\ S=-2su,\ T=2rt,\ U=ru+st$ if $ru-st=-1$.
\end{itemize}

So the double covering defined by $\varphi_1$ is the sapphire determined by the matrix
$\begin{pmatrix}
          ru+st & 2su \\
          -2rt & -(ru+st)
        \end{pmatrix}$
if $ru-st=1$, or the sapphire determined by
$\begin{pmatrix}
          -(ru+st) & -2su \\
          2rt & ru+st
        \end{pmatrix}$ if $ru-st=-1$. Moreover it is clear that both matrices describe the same 3-manifold.

\subsubsection*{Case IV: $\varphi_4(a)=\overline{1}$, $\varphi_4(b)=\overline{1}$, $\varphi_4(c)=\overline{0}$}

In this case $s$ must be even and $u$ odd, we denote $s=2k$ and $u=2l-1$. Taking $\{1,b\}$ as a Schreier system for right cosets of the kernel ($H$) of $\varphi_2$, we obtain  $\{ab^{-1},c,ba, b^{2}, bcb^{-1}\}$ as a system of generators and the following relations:

\begin{itemize}
\item $aba^{-1}b$;
\item $c^{2}a^{-2r}b^{-s}$;
\item $ ca^{2t}b^{u}c^{-1}a^{2t}b^{u} $;
\item $ b^{-1}aba^{-1}b^{2} $;
\item $ b^{-1}c^{2}a^{-2r}b^{-s+1} $;
\item $b^{-1}ca^{2t}b^{u}c^{-1}a^{2t}b^{u+1}  $.
\end{itemize}

Taking  $\alpha=ab^{-1},\beta=c,\gamma=ba,\delta=b^{2},\lambda=bcb^{-1}$,  we obtain
$$
H=\left\langle
\begin{array}
[c]{c}%
\alpha,\beta,\gamma,\delta,\lambda\;;\;\alpha\delta\gamma^{-1}\delta=\beta
^{2}\left(  \alpha\gamma\right)  ^{-r}\delta^{-k}=\beta\left(  \alpha
\gamma\right)  ^{t}\delta^{l-1}\lambda^{-1}\alpha^{-1}\left(  \alpha
\gamma\right)  ^{t}\alpha\delta^{l}=\gamma\alpha^{-1}\\
=\lambda^{2}\gamma\left(  \alpha\gamma\right)  ^{-r}\gamma^{-1}\delta
^{-k}=\lambda\gamma\left(  \alpha\gamma\right)  ^{t}\gamma^{-1}\delta^{l}%
\beta^{-1}\left(  \alpha\gamma\right)  ^{t}\delta^{l-1}=1
\end{array}
\right\rangle.$$

Setting $a_1=\gamma$, $b_1=\delta$ and $c_1=\beta$, we  have:
$$ H=< a_1,b_1,c_1\ |\ a_1b_1a_1^{-1}b_1=c_1^{2}a_1^{-2R}b_1^{S}= c_1a_1^{2T}b_1^{U}c_1^{-1}a_1^{2T}b_1^{U}=1>$$
where $R=r, S=k=\dfrac{s}{2}, T=2t$ and  $U=u$.

So the double covering defined by $\varphi_2$ is the sapphire determined by the matrix:
$$\begin{pmatrix}
          r & \dfrac{s}{2} \\
          2t & u
        \end{pmatrix}.$$

\subsubsection*{Case V: $\varphi_5(a)=\overline{1}$, $\varphi_5(b)=\overline{0}$, $\varphi_5(c)=\overline{1}$}

This case was solved in \cite{daci-peter}. The double covering is a torus bundle with Anosov gluing map given by the matrix:
$$\begin{pmatrix}
           ru + ts & -2rt \\
               -2su & ru + ts
           \end{pmatrix}.$$

  Let us recall that the fundamental group of this double covering corresponds to  the subgroup generated by $\{a^2, b, a^{-1}c\}$.

\subsubsection*{Case VI: $\varphi_6(a)=\overline{0}$, $\varphi_6(b)=\overline{1}$, $\varphi_6(c)=\overline{1}$}

In this case $s$ must be even and $u$ odd, we denote $s=2k$ and $u=2l-1$. Taking $\{1,b\}$ as a Schreier system for right cosets of the kernel ($H$) of $\varphi_2$,  we obtain  $\{a, cb^{-1}, bab^{-1}, b^2, bc\}$ as a system of generators, and the following relations:

\begin{itemize}
\item $aba^{-1}b$;
\item $c^{2}a^{-2r}b^{-s}$;
\item $ ca^{2t}b^{u}c^{-1}a^{2t}b^{u} $;
\item $ b^{-1}aba^{-1}b^{2} $;
\item $ b^{-1}c^{2}a^{-2r}b^{-s+1} $;
\item $b^{-1}ca^{2t}b^{u}c^{-1}a^{2t}b^{u+1}  $.
\end{itemize}

Taking  $\alpha=a,\beta=cb^{-1},\gamma=bab^{-1},\delta=b^{2},\lambda=bc$,   we obtain
$$
H=\left\langle
\begin{array}
[c]{c}%
\alpha,\beta,\gamma,\delta,\lambda\;;\;\alpha\gamma^{-1}\delta=\beta
\lambda\alpha^{-2r}\delta^{-k}=\beta\gamma^{2t}\delta^{l}\lambda^{-1}%
\gamma^{2t}\delta^{l}=\gamma\delta\alpha^{-1}=\\
\lambda\beta\gamma^{-2r}\delta^{-k}=\lambda\alpha^{2t}\delta^{l-1}\beta
^{-1}\alpha^{2t}\delta^{l-1}=1
\end{array}
\right\rangle.$$

Setting $a_1=\gamma$, $b_1=\delta$ and $c_1=\beta\alpha^{2t}\delta^{l}$, we  have:
$$ H=< a_1,b_1,c_1\ |\ a_1b_1a_1^{-1}b_1=c_1^{2}a_1^{-2R}b_1^{S}= c_1a_1^{2T}b_1^{U}c_1^{-1}a_1^{2T}b_1^{U}=1>$$
where $R=r, S=k=\dfrac{s}{2}, T=2t$ and  $U=u$.

So the double covering defined by $\varphi_2$ is the sapphire determined by the matrix:
$$\begin{pmatrix}
          r & \dfrac{s}{2} \\
          2t & u
        \end{pmatrix}.$$

\subsubsection*{Case VII: $\varphi_7(a)=\overline{1}$, $\varphi_7(b)=\overline{1}$, $\varphi_7(c)=\overline{1}$}

In this case $s$ must be even and $u$ odd, we denote $s=2k$ and $u=2l-1$.  Taking $\{1,b\}$ as a Schreier system for right cosets of the kernel ($H$) of $\varphi_2$,  we obtain  $\{ab^{-1}, cb^{-1}, ba, b^2, bc\}$ as a system of generators and the following relations:

\begin{itemize}
\item $aba^{-1}b$;
\item $c^{2}a^{-2r}b^{-s}$;
\item $ ca^{2t}b^{u}c^{-1}a^{2t}b^{u} $;
\item $ b^{-1}aba^{-1}b^{2} $;
\item $ b^{-1}c^{2}a^{-2r}b^{-s+1} $;
\item $b^{-1}ca^{2t}b^{u}c^{-1}a^{2t}b^{u+1}  $.
\end{itemize}

Taking  $\alpha=ab^{-1},\beta=cb^{-1},\gamma=ba,\delta=b^{2},\lambda=bc$, we obtain
$$
H=\left\langle
\begin{array}
[c]{c}%
\alpha,\beta,\gamma,\delta,\lambda\;;\;\alpha\delta\gamma^{-1}\delta
=\beta\lambda\left(  \alpha\gamma\right)  ^{-r}\delta^{-k}=\beta\gamma\left(
\alpha\gamma\right)  ^{t}\gamma^{-1}\delta^{l}\lambda^{-1}\gamma\left(
\alpha\gamma\right)  ^{t}\gamma^{-1}\delta^{l}=\gamma\alpha^{-1}\\
=\lambda\beta\gamma\left(  \alpha\gamma\right)  ^{-r}\gamma^{-1}\delta
^{-k}=\lambda\left(  \alpha\gamma\right)  ^{t}\delta^{l-1}\beta^{-1}\left(
\alpha\gamma\right)  ^{t}\delta^{l-1}=1
\end{array}
\right\rangle.$$

Setting $a_1=\gamma$, $b_1=\delta$ and $c_1=\beta\gamma^{2t}\delta^{l}$,  we  have:
$$ H=< a_1,b_1,c_1\ |\ a_1b_1a_1^{-1}b_1=c_1^{2}a_1^{-2R}b_1^{S}= c_1a_1^{2T}b_1^{U}c_1^{-1}a_1^{2T}b_1^{U}=1>$$
where $R=r, S=k=\dfrac{s}{2}, T=2t$ and  $U=u$.

So the double covering defined by $\varphi_2$ is the sapphire determined by the matrix:
$$\begin{pmatrix}
          r & \dfrac{s}{2} \\
          2t & u
        \end{pmatrix}.$$

We  will present the conclusions of the above calculations 
in the following table. If  $S_A$ is a Sol 3-manifold  that is a sapphire constructed using the matrix $A=\begin{pmatrix}
          r & s \\
          t & u
        \end{pmatrix}$, in the table we give matrices that describe the sapphire manifolds and the torus bundles, respectively, of the  double coverings of $S_A$. To simplify the table, in the case that  the double covering is a sapphire, we used the classification given by Morimoto \cite{morimoto} in Theorem 1 to choose the matrices. We recall that we are assuming $r,s,t$ and $u$ positive numbers and $r\le u$.

\begin{center}
\begin{small}
\label{tablesapphire}
\begin{tabular}{|c|c|c|c|c|c|}
\hline\hline
Case & Hom. &  Sapphire &  Sol torus bundle\\
\hline
I&$\varphi_1\left\{
\begin{array}[c]{ll}
a\mapsto \overline{1}\\
b\mapsto \overline{0}\\
c\mapsto \overline{0}
\end{array}
\right.
$ & $\begin{pmatrix}
          ru+st & 2rs \\
          2tu & ru+st
        \end{pmatrix}$ &  --  \\
\hline
II&$\varphi_2\left\{
\begin{array}[c]{ll}
a\mapsto \overline{0}\\
b\mapsto \overline{1}\\
c\mapsto \overline{0}
\end{array}
\right.$
  ($s$ even)
  &$\begin{pmatrix}
          r & \dfrac{s}{2} \\
          2t & u
        \end{pmatrix}$ & -- \\
\hline
III&$\varphi_3\left\{
\begin{array}[c]{ll}
a\mapsto \overline{0}\\
b\mapsto \overline{0}\\
c\mapsto \overline{1}
\end{array}
\right.
$ & $\begin{pmatrix}
          ru+st & 2su \\
          2rt & ru+st
        \end{pmatrix}$   &--  \\
\hline
IV&$\varphi_4\left\{
\begin{array}[c]{ll}
a\mapsto \overline{1}\\
b\mapsto \overline{1}\\
c\mapsto \overline{0}
\end{array}
\right.
$ ($s$ even)  &$\begin{pmatrix}
          r & \dfrac{s}{2} \\
          2t & u
        \end{pmatrix}$ &--  \\
\hline
V&$\varphi_5\left\{
\begin{array}[c]{ll}
a\mapsto \overline{1}\\
b\mapsto \overline{0}\\
c\mapsto \overline{1}
\end{array}
\right.
$  &-- & $\begin{pmatrix}
           ru + ts & -2rt \\
               -2su & ru + ts
           \end{pmatrix}$ \\
\hline
VI&$\varphi_6\left\{
\begin{array}[c]{ll}
a\mapsto \overline{0}\\
b\mapsto \overline{1}\\
c\mapsto \overline{1}
\end{array}
\right.
$  ($s$ even)  &$\begin{pmatrix}
          r & \dfrac{s}{2} \\
          2t & u
        \end{pmatrix}$ &--  \\
\hline
VII&$\varphi_7\left\{
\begin{array}[c]{ll}
a\mapsto \overline{1}\\
b\mapsto \overline{1}\\
c\mapsto \overline{1}
\end{array}
\right.
$  ($s$ even)  &$\begin{pmatrix}
          r & \dfrac{s}{2} \\
          2t & u
        \end{pmatrix}$ &--  \\
\hline
\end{tabular}
\end{small}
\end{center}

Recall from Corollary 2.3 in \cite{HaGoZv} that {\it two  classes $x_1\in H^{1}(W_1,\Z)$, $x_2\in H^{1}(W_2,\Z)$  are equivalent if there is a homotopy equivalence $h:W_1 \to W_2$  such that the induced homomorphism by $h$ on $H^1$ maps $x_2$ to $x_1$}.
The above equivalence relation is closely related to  an equivalence relation of involutions, which we give at the beginning of the next section.

Now we classify the   $\varphi's$ given by the table which are equivalent.

By considering the isomorphism which sends $a\to ab$, $b\to b$ and $c \to c$, we obtain that case $II$ is equivalent to case $IV$, and case $VI$ is equivalent to case $VII$. The following proposition shows that all of these  cases are in fact equivalent.

\begin{proposition}\label{equiv}  The cases $II$, $IV$, $VI$ and $VII$ are equivalent.
\end{proposition}

\begin{proof} The above considerations   show that case $II$ is equivalent to case $IV$ and case $VI$ is equivalent to case $VII$. So to proof  the proposition, it suffices to construct an automorphism
$\theta: \pi_1(K(r,s, t,u)) \to \pi_1(K(r,s,t,u))$ where $s$ is even, such that $\theta(\varphi_2)=\varphi_6$, i.e. case $II$ is equivalent to  case $VI$.

 Recall the presentation
\begin{equation}\label{sapphire}
\pi_1(N_{A})\cong \langle a,b,c\ |aba^{-1}b= c^{2}a^{-2r}b^{s}=ca^{2t}b^{u}c^{-1}a^{2t}b^{u}=1 \rangle.
\end{equation}

\noindent We will suppose that $det(A)=1$. Let $\theta(a)=a$, $\theta(b)=b$ and
$\theta(c)= w(a^2, b)c=wc$ be such that  the exponen sum s of $b$ is odd. If such then  $\theta$ exists,    $\theta(\varphi_2)=\varphi_6$ and the result follows. So let us show that $\theta$
preserves the relations. The first relation $aba^{-1}b=1$
is clearly preserved by $\theta$. The  third relation  $ca^{2t}b^{u}c^{-1}a^{2t}b^{u}=1$
is also  preserved since  the image of the relation is
$wca^{2t}b^{u}c^{-1}w^{-1}a^{2t}b^{u}=w(a^{2t}b^{u})^{-1}w^{-1}a^{2t}b^{u}$. But the subgroup given by
$<w, a^{2t}, b^u>$ is isomorphic to $\Z+\Z$ and it follows that  $w(a^{2t}b^{u})^{-1}w^{-1}a^{2t}b^{u}=1$.  So it suffices  to find $w$ with the above properties above and  the relation $c^{2}=b^{s}a^{2r}$ preserved by $\theta$. Thus $wcwc=wcwc^{-1}c^2=b^{s}a^{2r}$  or
$cwc^{-1}=w^{-1}b^{s}a^{2r}w^{-1}c^{-2}=w^{-1}b^{s}a^{2r}a^{-2r}b^{-s}=w^{-1}$. Hence  it suffices to find $w$ an eigenvector for the eigenvalue $-1$ (for the action of $c$ by conjugation on the subgroup $<w, a^{2t}, b^u>$) such that  the exponent sum of $b$ is odd.
 From \cite{daci-peter}, we know that the  matrix of the automorphism of the subgroup $<a^2, b>$ given
 by conjugation by $a^{-1}c$, is the matrix

$$\begin{pmatrix}
           ru + ts & -2rt \\
               -2su & ru + ts
           \end{pmatrix}
           $$
           
$$ \qquad \left(\mbox{equivalent, by the Morimoto classification, to} \begin{pmatrix}
           ru + ts & 2rt \\
               2su & ru + ts
           \end{pmatrix}\right) .$$

\noindent So the matrix of conjugation by $c$ is

$$\begin{pmatrix}
           ru + ts & -2rt \\
               2su & -ru - ts
           \end{pmatrix}.$$

Now we will show that this integral matrix has eigenvalue -1 and an eigenvector corresponding to the eigenvalue -1 where the exponent of $b$ is odd. The first part follows because the determinant of the matrix

$$\begin{pmatrix}
           ru + ts+1 & -2rt \\
                 2su & -ru - ts+1
           \end{pmatrix}$$

\noindent is zero, so we have an eigenvector for the eigenvalue -1. It remains to show that there is one such eigenvector with the exponent sum of $b$ odd. Let $(x,y)$ be an eigenvector. We have
$$  (ru + ts+1)x - (2rt)y=0.$$
But  $ru + ts+1= (ru - ts+1)+ 2ts=2+2ts$, where the last equality follows from $det(A)=1$.

So we  obtain  $(1+st)x=rty$, where $1+st$ is odd since $s$ is even. Therefore there is a solution with $y$ odd and the result follows if $det(A)=1$.

If $det(A)=-1$, $A=\begin{pmatrix}
           r & s \\
           t & u
           \end{pmatrix}$, we can chose $A'=\begin{pmatrix}
           -r & s \\
           -t & u
           \end{pmatrix}$ that determines the same sapphire manifold (by Morimoto's classification) with $det(A') =1$. Using variables $a,b$ and $c$ to describe the presentation of $\pi_1(S_A)$ and $x,y$ and $z$ for $\pi_1(S_{A'})$, we  see that the isomorphism $\gamma$ that sends $a\to x^{-1}$, $b\to y$ and $c\to z$ is such that the composition $\gamma^{-1}\varphi\gamma$ sends $\varphi_ 2$ to $\varphi_6$ for the presentation of $\pi_1(S_{A})$ and the result follows. \qed
           
 \end{proof}

So we  obtain the following result.

\begin{proposition}  For the homomorphisms $\varphi_i$ given by the above table,  we have:
\begin{enumerate}
\item The subset $\{\varphi_2, \varphi_4, \varphi_6, \varphi_7\}$  (when it exists)
is  a single equivalence class.
\item The subset $\{\varphi_5\}$ is a single  equivalence class.
\item The homomorphism $\varphi_1$ is not equivalent to $\varphi_3$ if $r\ne u$ and  $r\ne -u$. Therefore in this case, $\{\varphi_1\}$ and $\{\varphi_3\}$ are two distinct
equivalence classes. If either $r=u$ or $r=-u$ then the two double coverings which arise from $\varphi_1$, $\varphi_3$, respectively, are homeomorphic.
\end{enumerate}
\end{proposition}

\begin{proof} Part 1) follows from  Proposition \ref{equiv}.

Observe that if two pairs $(S_A, \alpha)$, $(S_A, \beta)$ are equivalent, then the corresponding double covering are homeomorphic. Part 2) follows from the fact that the double covering which corresponds to the case  $\{(S_A,\varphi_5)\}$ is a torus bundle, so cannot be equivalent to any other case.

Part 3) If the double coverings corresponding to the two cases are homeomorphic, then the two matrices are related as  in Theorem 1. But this implies that $|2rs|=|2su|$ or $r=\pm u$. So the result follows. \qed

 \end{proof}

{\bf Remark} We do not know if for either $r=u$ or $r=-u$ the two pairs $(S_A, \varphi_1)$ and $(S_A,\varphi_3)$ are equivalent or not.

\section{The classification of free involutions and the Borsuk-Ulam theorem}

In this section we classify the involutions of the sapphire manifolds with Sol geometry, and we compute the values of $n$ for which the 
triple $(M, \tau; \R^n)$ satisfies the Borsuk-Ulam property, for all possible free involutions.

We say that {\it two involutions $\tau_1$ on $M_1$ and $\tau_2$ on $M_1$ are equivalent if there is a homotopy equivalence $h$ between the orbit space such that the induced homomorphism by $h$ on $H^1$ maps the characteristic class of one involution to the characteristic class of the other involution}. For more about the relation between involutions and characteristic classes,  see \cite{HaGoZv}. 

We begin this section  by looking this at the triples  $(M, \tau; \R^2)$. The result in this case uses a general fact about the Borsuk-Ulam property which is independent of the involution. More precisely, for any pair $(M, \tau)$ where $M$ is a sapphire we have:

\begin{proposition}\label{r2} Any triple  $(M, \tau; \R^2)$ satisfies the Borsuk-Ulam property.
\end{proposition}

\begin{proof} From Morimoto, $H_1(M,\Z)$ is finite. The result then  follows immediatly from \cite[Corollary 3.3]{BauHaGoZv1}. \qed
\end{proof}

Let $S_A$  (with $A=\begin{pmatrix}
           a & b \\
           c & d
           \end{pmatrix}$) be a sapphire Sol manifold. Using the results of table~\ref{tablesapphire} (page \pageref{tablesapphire}), we may  whether  such a manifold is the double covering of another Sol manifold associated with some homomorphism.

In the $\varphi_1$ case we note that we must have $a=d$, $a$ odd (this implies $det(A)=1$) and $b$ and $c$ even. The following Proposition shows that these conditions are sufficient to guarantee that $S_A$ is the double covering of some Sol manifold which corresponds to the kernel of the homomorphism $\varphi_1$.

Given an integer $x$ and a prime $p$, denote the largest integer such that $p^{|x|_p}$
divides $x$ by $|x|_p$. If $p$ does not divide $x$ then we  define this number to be  zero.

\begin{proposition}\label{SapphireDown1} If $a=d$, $a$ odd, $b$ and $c$ are even then $S_A$ is the double covering of two sapphire Sol 
manifolds $S_B$ ($B=\begin{pmatrix}
           r & s \\
           t & u
           \end{pmatrix} \in M_2(\mathbb Z)$), associated to $\varphi_1$,  i.e. (see table~\ref{tablesapphire}) the system
\begin{equation}
a=ru+ts,\ b=2rs,\ c=2tu\mbox{ and } d=ru+st,\label{sistema1}
\end{equation}
\noindent admits 2 solutions with $r,s,t$ and $u$ positive.

More precisely, $B$ is one of the following matrices:

$$
\left(
\begin{array}
[c]{cc}%
\prod\limits_{j=1}^{n}q_{j}^{\left\vert \frac{b}{2}\right\vert _{q_{j}}} &
\prod\limits_{i=1}^{m}p_{i}^{\left\vert \frac{b}{2}\right\vert
_{p_{i}}}\\
\prod\limits_{i=1}^{m}p_{i}^{\left\vert
\frac{c}{2}\right\vert _{p_{i}}} & \prod\limits_{j=1}^{n}%
q_{j}^{\left\vert \frac{c}{2}\right\vert _{q_{j}}}%
\end{array}
\right)
\qquad
\left(
\begin{array}
[c]{cc}%
\prod\limits_{i=1}^{m}p_{i}^{\left\vert \frac{b}{2}\right\vert _{p_{i}}} &
\prod\limits_{j=1}^{n}q_{j}^{\left\vert \frac{b}{2}\right\vert
_{q_{j}}}\\
\prod\limits_{j=1}^{n}q_{j}^{\left\vert
\frac{c}{2}\right\vert _{q_{j}}} & \prod\limits_{i=1}^{m}%
p_{i}^{\left\vert \frac{c}{2}\right\vert _{p_{i}}}%
\end{array}
\right)
$$

where $p_i$ runs over the set of all divisors of $(a+1)/2$, and $q_i$ runs over the set  of all divisors of $(a-1)/2$. 

\end{proposition}

\begin{proof}

We can write $A$ as:

\begin{equation}
A=\left(
\begin{array}
[c]{ll}%
\alpha & 2\beta\\
2\gamma & \alpha
\end{array}
\right)  \in M_{2\times2}\left(  \mathbb{Z-}\left\{  0\right\}  \right)
\label{matrizA}%
\end{equation}
where $\alpha,\beta,\gamma\in\mathbb{N-}\left\{  0\right\}  $, $\alpha$ are odd, and $\det A=\pm 1$.

\bigskip

We will verify that such conditions are in fact sufficient. Moreover we will show that there exist exactly two matrices (with positive entries) $B$ satisfying such properties, and they are in two equivalence classes for the relation described by Morimoto.

More precisely, there are only two different homeomorphic classes of sapphire that admit $S_{A}$ as a double covering.

\bigskip

Suppose that $A$ satisfies \ref{matrizA}. Observe that
\[
1=\det A=\alpha^{2}-4\beta\gamma\Longrightarrow\beta\gamma=\frac{\alpha+1}%
{2}.\frac{\alpha-1}{2}%
\]

So $det(A)=\alpha^2-4\beta\gamma=1$ ( $det(A)=-1$  is not possible because $\alpha=2k+1\Rightarrow \alpha^2-4\beta\gamma=4(n^2+n-\beta\gamma)+1$).

If there exists a matrix $B$ satisfying \ref{sistema1}, we must have

\begin{equation}
rstu=\beta\gamma=\frac{\alpha+1}{2}.\frac{\alpha-1}{2}.\label{rstu1}%
\end{equation}

Now we define $R=\frac{\alpha+1}{2}$, $S=\frac{\alpha-1}{2}\in\mathbb{Z}$. So $R,S\neq0$ since
\[
RS=rstu\neq0
\]
and
\[
R=\prod\limits_{i=1}^{m}p_{i}^{\left\vert R\right\vert _{p_{i}%
}}\qquad \mbox {,} \qquad S=\prod\limits_{j=1}^{n}q_{j}^{\left\vert
S\right\vert _{q_{j}}}%
\]
where $p_{i}$ and $q_{j}$ are positive primes and $p_{i}\neq q_{j}$, for all $i$ and $j$ ($R$ and $S$ are consecutive integers).

\bigskip

Let  $x=st\in\mathbb{Z}-\left\{  0\right\}  $. It follows from \ref{sistema1} and \ref{rstu1} that
\[
\left( \alpha-x\right)  x=rux=RS
\]
and so
\[
x^2-\alpha x+RS=0
\]

\bigskip

An easy calculation shows that $x=R$ or $x=S$:

\bigskip

\textbf{Case:} $ts=x=R$ ($\Longrightarrow$ $ru=S$)

\bigskip

In such a situation,  we may write
\[
t=\prod\limits_{i=1}^{m}p_{i}^{\left\vert t\right\vert _{p_{i}%
}}\qquad \mbox{and}\qquad s=\prod\limits_{i=1}^{m}p_{i}^{\left\vert
s\right\vert _{p_{i}}}%
\]%
\[
r=\prod\limits_{j=1}^{n}q_{j}^{\left\vert r\right\vert _{q_{j}%
}}\qquad \mbox{and}\qquad u=\prod\limits_{j=1}^{n}q_{j}^{\left\vert
u\right\vert _{q_{j}}}%
\]
where
\[
\left\vert t\right\vert _{p_{i}}+\left\vert s\right\vert _{p_{i}}=\left\vert
R\right\vert _{p_{i}}\qquad \mbox{and}\qquad\left\vert r\right\vert _{q_{j}}+\left\vert
u\right\vert _{q_{j}}=\left\vert S\right\vert _{q_{j}}%
\]%

Using \ref{sistema1},  we obtain
\[
\left\vert r\right\vert _{q_{j}}=\left\vert \beta\right\vert _{q_{j}}\qquad
\mbox{and}\qquad\left\vert s\right\vert _{p_{i}}=\left\vert \beta\right\vert _{p_{i}}%
\]%
\[
\left\vert t\right\vert _{p_{i}}=\left\vert \gamma\right\vert _{p_{i}}\qquad
\mbox{and}\qquad\left\vert u\right\vert _{q_{j}}=\left\vert \gamma\right\vert _{q_{j}}%
\]%

Then

\[
B=\left(
\begin{array}
[c]{cc}%
\prod\limits_{j=1}^{n}q_{j}^{\left\vert \beta\right\vert
_{q_{j}}} & \prod\limits_{i=1}^{m}%
p_{i}^{\left\vert \beta\right\vert _{p_{i}}}\\
\prod\limits_{i=1}^{m}%
p_{i}^{\left\vert \gamma\right\vert _{p_{i}}} & \prod\limits_{j=1}^{n}q_{j}^{\left\vert \gamma\right\vert _{q_{j}}}%
\end{array}
\right)
\]

\bigskip

\textbf{Case:} $ts=x=S$ ($\Longrightarrow$ $ru=R$)

\bigskip

Now
\[
t=\prod\limits_{j=1}^{n}q_{j}^{\left\vert t\right\vert _{q_{j}%
}}\qquad \mbox{and}\qquad s=\prod\limits_{j=1}^{n}q_{j}^{\left\vert
s\right\vert _{q_{j}}}%
\]%
\[
r=\prod\limits_{i=1}^{m}p_{i}^{\left\vert r\right\vert _{p_{i}%
}}\qquad \mbox{and}\qquad u=\prod\limits_{i=1}^{m}p_{i}^{\left\vert
u\right\vert _{p_{i}}}%
\]
where%
\[
\left\vert t\right\vert _{q_{j}}+\left\vert s\right\vert _{q_{j}}=\left\vert
S\right\vert _{q_{j}}\qquad \mbox{and}\qquad\left\vert r\right\vert _{p_{i}}+\left\vert
u\right\vert _{p_{i}}=\left\vert R\right\vert _{p_{i}}%
\]%

Again using \ref{sistema1}, we obtain
\[
\left\vert r\right\vert _{p_{i}}=\left\vert \beta\right\vert _{p_{i}}\qquad
\mbox{and}\qquad\left\vert s\right\vert _{q_{j}}=\left\vert \beta\right\vert _{q_{j}}%
\]%
\[
\left\vert t\right\vert _{q_{j}}=\left\vert \gamma\right\vert _{q_{j}}\qquad
\mbox{and}\qquad\left\vert u\right\vert _{p_{i}}=\left\vert \gamma\right\vert _{p_{i}}%
\]%

Then

\[
B=\left(
\begin{array}
[c]{cc}%
\prod\limits_{i=1}^{m}p_{i}^{\left\vert \beta\right\vert
_{p_{i}}} & \prod\limits_{j=1}^{n}%
q_{j}^{\left\vert \beta\right\vert _{q_{j}}}\\
\prod\limits_{j=1}^{n}%
q_{j}^{\left\vert \gamma\right\vert _{q_{j}}} & \prod\limits_{i=1}^{m}p_{i}^{\left\vert \gamma\right\vert _{p_{i}}}%
\end{array}
\right)
\]

and we get one possible solution:
\[
B=\left(
\begin{array}
[c]{cc}%
\prod\limits_{i=1}^{m}p_{i}^{\left\vert \beta\right\vert _{p_{i}}} &
\prod\limits_{j=1}^{n}q_{j}^{\left\vert \beta\right\vert
_{q_{j}}}\\
\prod\limits_{j=1}^{n}q_{j}^{\left\vert
\gamma\right\vert _{q_{j}}} & \prod\limits_{i=1}^{m}%
p_{i}^{\left\vert \gamma\right\vert _{p_{i}}}%
\end{array}
\right) 
\]

\bigskip

Now observe that the matrix in the second case is not equivalent to the matrix in the first case. \qed
\end{proof}

 We also have a similar Proposition for the solution of the case of double coverings of type III of the table. The system of equations is that given above where we simply replace $r$ by $u$ and $u$ by $r$. Therefore the solution is as above where the roles of $r$ and $u$ are interchanged. More precisely:

\begin{proposition}\label{SapphireDown3} If $a=d$, $a$ odd, $b$ and $c$ are even then $S_A$ is the double covering of two sapphire Sol manifolds $S_B$ ($B=\begin{pmatrix}
           r & s \\
           t & u
           \end{pmatrix} \in M_2(\mathbb Z)$), associated to $\varphi_3$,  i.e. (see table~\ref{tablesapphire}) the system
\begin{equation}
a=ru+ts,\ b=2su,\ c=2rt\mbox{ and } d=ru+st,\label{sistema3}
\end{equation}
\noindent admits 2 solutions with $r,s,t$ and $u$ positive.

More precisely,  $B$ is one of the following matrices:

$$
\left(
\begin{array}
[c]{cc}%
\prod\limits_{j=1}^{n}q_{j}^{\left\vert \frac{c}{2}\right\vert _{q_{j}}} &
\prod\limits_{i=1}^{m}p_{i}^{\left\vert \frac{b}{2}\right\vert
_{p_{i}}}\\
\prod\limits_{i=1}^{m}p_{i}^{\left\vert
\frac{c}{2}\right\vert _{p_{i}}} & \prod\limits_{j=1}^{n}%
q_{j}^{\left\vert \frac{b}{2}\right\vert _{q_{j}}}%
\end{array}
\right)
\qquad
\left(
\begin{array}
[c]{cc}%
\prod\limits_{i=1}^{m}p_{i}^{\left\vert \frac{c}{2}\right\vert _{p_{i}}} &
\prod\limits_{j=1}^{n}q_{j}^{\left\vert \frac{b}{2}\right\vert
_{q_{j}}}\\
\prod\limits_{j=1}^{n}q_{j}^{\left\vert
\frac{c}{2}\right\vert _{q_{j}}} & \prod\limits_{i=1}^{m}%
p_{i}^{\left\vert \frac{b}{2}\right\vert _{p_{i}}}%
\end{array}
\right)
$$

\noindent where $p_i$ runs over the set of all divisors of $(a+1)/2$ and $q_i$ runs over the set  of all divisors of $(a-1)/2$. 

\end{proposition} 

We note that if $\begin{pmatrix}
           x & y \\
           z & w
           \end{pmatrix}$ is one of the solutions presented in Proposition~\ref{SapphireDown1} then the other solution, in the same proposition is $\begin{pmatrix}
           y & x \\
           w & z
           \end{pmatrix}$ and the two solutions of Proposition~\ref{SapphireDown3} are $\begin{pmatrix}
           w & y \\
           z & x
           \end{pmatrix}$ and $\begin{pmatrix}
           z & x \\
           w & y
           \end{pmatrix}$. So if $x\neq w$ and $y\neq z$ only two of then appears in the table~\ref{tablesapphire}, in two different position of the table. In the case that $x=w$ or $y=z$ two of then appears in the table~\ref{tablesapphire} but one of then appears twice, in two different positions of the table.

We are now in a position to state the main result, which   classifies the equivalence classes of free involutions of a sapphire manifold, except in one case where we obtain a partial classification. As a Corollary we apply results from  J. Hillman, \cite{hillman} to determine which triples $(S_A, \tau, \R^3)$ have the  Borsuk-Ulam property (we already know that the Borsuk-Ulam property  holds for all triples $(S_A, \tau, \R^2)$ by Proposition~\ref{r2}). 

  It is   more convenient to state the main theorem without the condition that all entries of the matrix that determine the sapphire  manifold are positive. So we will not assume this hypothesis below. 

\begin{theorem}\label{main} Given a sapphire 3-manifold determined by the matrix

$$\begin{pmatrix}
           a & b \\
           c & d
           \end{pmatrix}$$
\noindent we have:

I) If $c$ is odd the manifold does not admit involutions.

II) If $c,b$ are even we have two cases:

II-a) $|a|\neq |d|$, then the manifold admits one class of free involutions, and the quotient is homeomorphic to the sapphire determined by the matrix
$$\begin{pmatrix}
           a & 2b \\
         c/2 & d
           \end{pmatrix}.$$

II-b) $|a|=|d|$, then the manifold admits 3 distinct classes of involution and the quotients are the sapphire manifolds given by the matrices:
$$\begin{pmatrix}
           a & 2b \\
         c/2 & d,
           \end{pmatrix},\qquad 
           \begin{pmatrix}
           r & s \\
           t & u,
           \end{pmatrix},\qquad
           \begin{pmatrix}
           s & r \\
           u & t
           \end{pmatrix},$$

\noindent where $(r,s,t,u)$ is one of the solutions given by Proposition \ref{SapphireDown1}, if  $r\ne u$ and $s\ne t$. Otherwise  the manifold admits at least three distinct classes of involutions and at most five.

III) If $c$ is even and $b$ odd, then the manifold admits one class of free involutions and the quotient is homeomorphic to the sapphire determined by the matrix
$$\begin{pmatrix}
           a & 2b \\
         c/2 & d
           \end{pmatrix}.$$

 \end{theorem}

\begin{proof} First,  we note that given a matrix $A=\begin{pmatrix}
           a & b \\
           c & d
           \end{pmatrix}$ (with all entries positive) of a sapphire,  table~\ref{tablesapphire} describes 6 matrices of double coverings of $S_A$. 
           
           In fact   table~\ref{tablesapphire} can be constructed without the condition that all entries of the matrix be positive. In this case, if we apply the (new) table~\ref{tablesapphire} to another matrix $B$ such that $S_B$ is homeomorphic to $S_A$,  the 6 matrices obtained in this way 
           will be representatives of the same 6 manifolds (this  correspondence may change the position in the table).

Part I) follows easily from   table~\ref{tablesapphire}.

Part II) We just apply  table~\ref{tablesapphire}.  Then Proposition~\ref{equiv} (for II-a) and Propositions~\ref{equiv}, \ref{SapphireDown1} and \ref{SapphireDown3} (II-b).

Part III) Again we just use table~\ref{tablesapphire}. \qed

\end{proof}

\begin{corollary}\label{bumain} Given a sapphire 3-manifold $S_A$ determined by the matrix

$$\begin{pmatrix}
           a & b \\
           c & d
           \end{pmatrix}$$
\noindent the Borsuk-Ulam property holds for the triple $(S_A, \tau, \R^3)$ if and only if $c$ is even and $b$ is odd. Furthermore, in this case the manifold admits only one class of free involutions.  

 \end{corollary}

\begin{proof} If $c$ is odd there is no involution. For the other parts, it will be important to describe, among the 7 homomorphisms, those which factor through $\pi/\square(\pi)$. They are precisely those where $b$ is sent to zero. So $\varphi_2$,  $\varphi_4$, $\varphi_6$ and  $\varphi_7$ are those which do not factor though  $\pi/\square(\pi)$, so the cube of the associated cohomology class is nontrivial precisely when the element in position (1,2) is divisible by 2 but not by 4.

Apply \cite{hillman} section 5 item 11,  that the cube of cohomology class associated to $\varphi_i $ is nontrivial only in the Part III of Theorem~\ref{main}. \qed

\end{proof}

\section*{acknowledgement}

The second  author would like to think Michel Boileau for some fruitful discussions at an early stage of this work.

\end{document}